\documentclass[11pt]{article}

\usepackage[all]{xy}
\usepackage{fullpage}
\usepackage{amssymb}
\usepackage{amsmath}
\usepackage{amsthm}
\usepackage[all]{xy}
\usepackage{graphicx}
\usepackage{enumitem}
\usepackage{authblk}
\usepackage{hyperref}

\title{A note on the atomicity of arithmeticity}
\author[a,d,e]{Michael Hoefnagel}
\author[b,c,d]{Pierre-Alain Jacqmin\thanks{The second author is grateful to the FNRS for its generous support.}}
\affil[a]{\small{\textit{Mathematics Division, Department of Mathematical Sciences, Stellenbosch University, Private Bag X1 Matieland 7602, South Africa}}}
\affil[b]{\small{\textit{Institut de Recherche en Math\'ematique et Physique, Universit\'e catholique de Louvain, Chemin du Cyclotron~2, B 1348 Louvain-la-Neuve, Belgium}}}
\affil[c]{\small{\textit{Department of Mathematics, Royal Military Academy, Rue Hobbema~8, B 1000 Brussels, Belgium}}}
\affil[d]{\small{\textit{Centre for Experimental Mathematics, Department of Mathematical Sciences, Stellenbosch University, Private Bag X1 Matieland 7602, South Africa}}}
\affil[e]{\small{\textit{National Institute for Theoretical and Computational Sciences (NITheCS), South Africa}}}
\date{}

\newcommand{\lex}{\mathsf{lex}}
\newcommand{\reg}{\mathsf{reg}}
\newcommand{\C}{\mathbb{C}}
\renewcommand{\c}{\sqsubset}
\renewcommand{\implies}{\Rightarrow}
\newcommand{\equivalent}{\Leftrightarrow}
\newcommand{\Mal}{\mathsf{Mal}}
\newcommand{\Maj}{\mathsf{Maj}}
\newcommand{\Ari}{\mathsf{Ari}}
\newcommand{\mclex}{\mathsf{mclex}}
\newcommand{\mcreg}{\mathsf{mcreg}}
\newcommand{\Mclex}{\mathsf{Mclex}}
\newcommand{\Mcreg}{\mathsf{Mcreg}}
\newcommand{\Set}{\mathsf{Set}}
\newcommand{\Bool}{\mathsf{Bool}}
\newcommand{\op}{\mathsf{op}}
\newcommand{\matr}{\mathsf{matr}}
\newcommand{\M}{\mathsf{M}}
\newcommand{\N}{\mathsf{N}}

\providecommand{\keywords}[1]
{
  \small	
  \textbf{\textit{Keywords---}} #1
}

\providecommand{\classification}[1]
{
  \small	
  \textbf{\textit{2020 Mathematics Subject Classification---}} #1
}

\newtheorem{theorem}			     {Theorem}	    [section]
\newtheorem{proposition}  [theorem]	 {Proposition}
\newtheorem{corollary}	  [theorem]	 {Corollary}

\theoremstyle{definition}

\newtheorem{remark} 	  [theorem]  {Remark}

\newdir{ >}{@{}*!/-10pt/@{>} }

\begin{document}

\maketitle

\begin{abstract}
The main aim of this note is to show that, in the regular context, every matrix property in the sense of~\cite{Janelidze06} either implies the Mal'tsev property, or is implied by the majority property. When the regular category $\C$ is arithmetical, i.e., both Mal'tsev and a majority category, then we show that $\C$ satisfies every non-trivial matrix property.
\end{abstract}

\keywords{Arithmetical category, Mal'tsev category, majority category, matrix property, Mal'tsev condition.}

\classification{
18E13, 
08B05, 
18E08, 
08A05, 
06-08. 
}

\section{Introduction}

Consider an extended matrix of variables
$$\M=\left[ \begin{array}{ccc|c}
x_{11} & \cdots & x_{1m} & y_{1} \\
\vdots &        & \vdots & \vdots\\
x_{n1} & \cdots & x_{nm} & y_{n}\end{array} \right]$$
where the $x_{ij}$'s and the $y_i$'s are (not necessarily distinct) variables from $\{x_1,\dots,x_k\}$. We refer to the last column (the column of $y_i$'s in~$\M$) as the \emph{right column} of~$\M$, and every other column as a \emph{left column}. A \emph{row-wise interpretation} of $\M$ of type $(X_1,\dots, X_n)$ is a matrix of the form 
$$\left[ \begin{array}{ccc|c}
f_1(x_{11}) & \cdots & f_1(x_{1m}) & f_1(y_{1}) \\
\vdots      &   	 & \vdots      & \vdots     \\
f_n(x_{n1}) & \cdots & f_n(x_{nm}) & f_n(y_{n})\end{array} \right]$$
where the $f_i\colon\{x_1,\dots,x_k\} \rightarrow X_i$ are  functions. Given any relation $R \subseteq X_1 \times \cdots \times X_n$, we say that it is strictly $\M$-closed if for every row-wise interpretation $N$ of $\M$ of type $(X_1,\dots,X_n)$ if the left columns of $N$ are elements of $R$ then the right column of $N$ is also an element of~$R$. This set-theoretic property of relations can be internalised (via the Yoneda embedding) in any finitely complete category~$\C$, so that $\C$ is then said to have $\M$-closed relations if every internal relation in $\C$ is strictly $\M$-closed. We will also refer to the property of $\M$-closedness of internal relations in $\C$ as simply the \emph{matrix property}~$\M$.

Let us formulate this property in a way which we will use for the remainder of this paper. Given two morphisms $f$ and $g$ in a category $\C$ with the same codomain, we will write $f \c g$ if $f$ factors through~$g$. This relation on morphisms in $\C$ defines a \emph{cover relation} in the sense of \cite{Janelidze08, Janelidze09}. Then we can reformulate the matrix property $\M$ of a finitely complete category $\C$ in the following way: given any internal relation $r\colon R \rightarrowtail X_1 \times \cdots \times X_n$ in $\C$ and any row-wise interpretation 
$$N = \left[ \begin{array}{ccc|c}
c_1 & \cdots & c_m & y
\end{array} \right]$$
of the matrix $\M$ of type $(\hom(S,X_1), \dots, \hom(S,X_n))$ where the $c_i$'s are the left columns of $N$ and the $y$ is the right column of $N$ (viewed as morphisms $S \rightarrow X_1 \times \cdots \times X_n$), then, if for every $i \in \{1,\dots,m\}$ we have $c_i \c r$, then we have $y \c r$.

The first and most well-known example of a matrix property is the property of a finitely complete category to be a Mal'tsev category~\cite{CPP92,CLP91}, since a category $\C$ with finite limits is a Mal'tsev category if and only if every internal relation is \emph{difunctional} \cite{Riguet48}, i.e., every internal relation is strictly $\Mal$-closed where
$$\Mal = \left[ \begin{array}{ccc|c}
x_1 & x_2 & x_2 & x_1 \\
x_2 & x_2 & x_1 & x_1
\end{array} \right].$$

Similarly to what was done in~\cite{HJ22}, given integers $n,k > 0$ and $m\geqslant 0$, we write $\matr(n,m,k)$ for the set of all matrices with $n$ rows, $m+1$ columns and whose entries are in the set $\{x_1,\dots,x_k\}$. Then $\matr$ is the union of all such $\matr(n,m,k)$ for $n,k > 0$ and $m\geqslant 0$. Corresponding to a matrix $\M \in \matr$, we will write $\mclex\{\M\}$ for the collection of finitely complete categories which satisfy the matrix property~$\M$, and refer to these collections as \emph{matrix classes}. The collection of all such matrix classes is then denoted by $\Mclex$, i.e.,
$$\Mclex = \{\mclex\{\M\}\mid \M \in \matr\}$$
and it has a poset structure given by inclusion of matrix classes. We will also write $\Mclex[n,m,k]$ for the sub-poset of $\Mclex$ of matrix classes $\mclex\{\M\}$ determined by a matrix $\M$ in $\matr(n,m,k)$. Among the elements of the poset $\Mclex$ are two \emph{trivial} ones, i.e., the matrix class of preorders with a single isomorphism class and the matrix class of finitely complete preorders. These are respectively the bottom element and the unique atom of $\Mclex$. They are determined by the so-called \emph{trivial matrices} (see~\cite{HJJ22}). The top element of $\Mclex$, i.e., the matrix class of all finitely complete categories is called the \emph{anti-trivial} element and is determined by the so-called \emph{anti-trivial matrices}. The \emph{degenerate matrix classes} (respectively the \emph{degenerate matrices}) are the ones which are either trivial or anti-trivial.

As another example of a collection of categories determined by a matrix property, consider the matrix property corresponding to the matrix $\Maj$ where
$$\Maj = \left[ \begin{array}{ccc|c}
x_1 & x_1 & x_2 & x_1\\
x_1 & x_2 & x_1 & x_1\\
x_2 & x_1 & x_1 & x_1
\end{array} \right].$$
Then, $\mclex\{\Maj\}$ is the collection of all finitely complete majority categories~\cite{Hoefnagel19}. For another example, consider the matrix $\Ari$ where
$$\Ari = \left[ \begin{array}{ccc|c}
x_1 & x_2 & x_2 & x_1\\
x_2 & x_2 & x_1 & x_1\\
x_1 & x_2 & x_1 & x_1
\end{array} \right].$$
Then $\mclex\{\Ari\}$ is the collection of all finitely complete arithmetical categories as defined in~\cite{HJJ22}. Note that in the Barr-exact context~\cite{BGO71} with coequalisers this matrix property determines arithmetical categories as introduced first in~\cite{Pedicchio96} and later generalised to wider contexts in~\cite{Bourn01,GRT20}.

Some matrix properties imply others, which is to say that some matrix classes are contained (as sub-collections) in other matrix classes. For example, we have that $\mclex\{\Ari\} \subseteq \mclex\{\Maj\}$ and also $\mclex\{\Ari\} \subseteq \mclex\{\Mal\}$. Thus, the general question: is there a procedure for determining when a given matrix class contains another? This question has been recently answered in the paper~\cite{HJJ22}, where an algorithm was given which determines inclusions of the form $\mclex\{\N\} \subseteq \mclex\{\M\}$, i.e., which determines implications of matrix properties. Computer implementation of this algorithm allows us to determine, for relatively small $n,m,k$, the posets $\Mclex[n,m,k]$. For instance, Figure~\ref{figure 3,7,2} (which describes the same poset as Figure~2 in~\cite{HJJ22} and Figure~1 in~\cite{HJJW23}) gives a visual depiction of the poset of non-degenerate elements of $\Mclex[3,7,2]$ as obtained by the computer, where each integer entry $i$ corresponds to the variable $x_i$ and the shaded column is the right column in the representing matrix. Note that most of these matrix classes are not represented here by $3\times (7+1)$ matrices but up to duplication of rows and left columns, they can be turned to such matrices. One of the main results of~\cite{HJJW23} shows that among the non-trivial matrix classes in $\Mclex$ represented by a matrix with at most two variables (i.e., binary matrices) the matrix class $\mclex\{\Ari\}$ is the least. In fact, Figure~\ref{figure 3,7,2} already illustrates this fact, as the matrix class farthest to the left is $\mclex\{\Ari\}$. However, in $\Mclex$ we have non-trivial matrix classes which are strictly contained in $\mclex\{\Ari\}$.
\begin{figure}[htb!] 
	\centering
	\includegraphics[width=450pt]{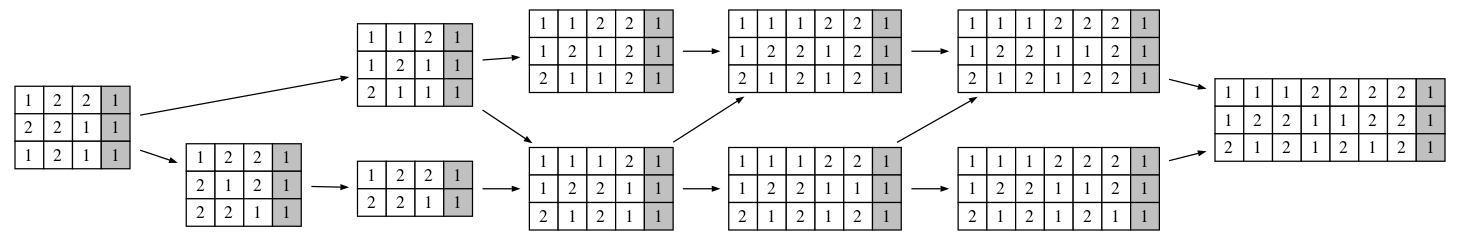}
	\caption{The poset of non-degenerate elements of $\mathsf{Mclex}[3,7,2]$. The matrix furthest to the left represents the matrix class $\mclex\{\Ari\}$ and the only matrix with two rows represents $\mclex\{\Mal\}$. The matrix which is just above this latter matrix represents the matrix class $\mclex\{\Maj\}$.}
	\label{figure 3,7,2}
\end{figure}

Let us now write $\mcreg\{\M\}$ for the collection of all regular categories~\cite{BGO71} satisfying the matrix property~$\M$, i.e., $\mcreg\{\M\}$ is the intersection of $\mclex\{\M\}$ and the collection of all regular categories. Such classes of regular categories we refer to as \emph{regular} matrix classes. We may then consider the poset $\Mcreg$ of all regular matrix classes ordered by inclusion, and ask the analagous question: is there an algorithm for determining whether or not $\mcreg\{\M\} \subseteq \mcreg\{\N\}$? As it stands, this question is still open, and it is known that $\mcreg\{\M\} \subseteq \mcreg\{\N\}$ need not imply that $\mclex\{\M\} \subseteq \mclex\{\N\}$ --- see Section~5 of~\cite{HJJ22}. 

In this paper we will show that the regular matrix classes corresponding to $\Mal$, $\Maj$ and $\Ari$ play a special role in $\Mcreg$. For one thing, we will show that every regular matrix class in $\Mcreg$ is either contained in $\mcreg\{\Mal\}$ or contains $\mcreg\{\Maj\}$. We will also show that among the non-trivial members of $\Mcreg$, the least is the regular matrix class corresponding to~$\Ari$. While, as shown in~\cite{HJJW23}, both results extend to the finitely complete context for binary matrices, we know they do not extend to that context in full generality.

\subsubsection*{Notation}

In order to simplify notation, we will borrow the notation of~\cite{HJ23}, and write $\M \implies_{\lex} \N$ if $\mclex\{\M\} \subseteq \mclex\{\N\}$ and likewise write $\M \implies_{\reg} \N$ if $\mcreg\{\M\} \subseteq \mcreg\{\N\}$.

\section{A strong majority property}

Given a category $\C$ with binary products, in what follows we will write $\pi_{i,j}$ for the two-fold projection $(\pi_i, \pi_j)\colon X_1 \times \cdots \times X_n \rightarrow X_i \times X_j$ determined by $\pi_i$ and~$\pi_j$, where $i,j \in \{1,2,\dots,n\}$. For a natural number $n \geqslant 3$, we define the following property on a finitely complete category~$\C$.

\begin{description}
\item[$(M_n)$] For any morphism $y\colon S \rightarrow X_1 \times \cdots \times X_n$ and any monomorphism $r\colon R \rightarrowtail X_1 \times \cdots \times X_n$, if $\pi_{i,j}y\c \pi_{i,j}r$ for any $i,j\in\{1,2,\dots,n\}$, then $y\c r$.
\end{description}
Note that for such $y\colon S \rightarrow X_1 \times \cdots \times X_n$, $r\colon R \rightarrowtail X_1 \times \cdots \times X_n$ and $i,j\in\{1,\dots,n\}$, if $\pi_{i,j}y\c \pi_{i,j}r$, then $\pi_{j,i}y\c \pi_{j,i}r$ and $\pi_{i,i}y\c \pi_{i,i}r$; so that in the above description of $(M_n)$, it is equivalent to ask $\pi_{i,j}y\c \pi_{i,j}r$ only for all $i,j\in\{1,\dots,n\}$ with $i<j$. As we will see shortly, the property $(M_n)$ for any integer $n\geqslant3$ is equivalent to a matrix property. Define a matrix $\M_n$ with $n$ rows, $m=\binom{n}{2}$ left columns $c_{i,j}$ indexed by all pairs of integers $(i,j)$ where $1 \leqslant i<j \leqslant n$ and whose right column is the column vector containing only the variable~$x_1$. Order the left columns $c_{i,j} < c_{i',j'}$ from left to right according to the lexicographic order $(i,j) < (i',j')$ on $\mathbb{N}^2$. In each column $c_{i,j}$ place a $x_1$ at the $i^{\textrm{th}}$ and $j^{\textrm{th}}$ entry. Then in each row, insert the variables $x_2,\dots, x_k$ (where $k = \binom{n-1}{2}+1$) in increasing order (of index) at each position which does not contain a~$x_1$. For example, in the case $n=3$, the matrix $M_3$ is nothing but
$$\M_3=\Maj = \left[ \begin{array}{ccc|c}
x_1 & x_1 & x_2 & x_1\\
x_1 & x_2 & x_1 & x_1\\
x_2 & x_1 & x_1 & x_1
\end{array} \right]$$
as defined in the Introduction. In the case $n=4$, we have
$$\M_4 = \left[ \begin{array}{cccccc|c}
x_1 & x_1 & x_1 & x_2 & x_3 & x_4 & x_1\\
x_1 & x_2 & x_3 & x_1 & x_1 & x_4 & x_1\\
x_2 & x_1 & x_3 & x_1 & x_4 & x_1 & x_1\\ 
x_2 & x_3 & x_1 & x_4 & x_1 & x_1 & x_1
\end{array} \right].$$

\begin{proposition}\label{proposition: Mn}
Let $n\geqslant 3$ be an integer. A finitely complete category $\C$ satisfies $(M_n)$ if and only if $\C$ has $\M_n$-closed relations, i.e., $\C$ satisfies the matrix property corresponding to~$\M_n$. 
\end{proposition}

\begin{proof}
Suppose that $\C$ satisfies $(M_n)$ and let $r\colon R\rightarrowtail X_1 \times \cdots \times X_n$ be any monomorphism and 
$$M'=\left[ \begin{array}{ccc|c}
x_{11} & \cdots & x_{1m} & y_{1} \\
\vdots &  		& \vdots & \vdots\\
x_{n1} & \cdots & x_{nm} & y_{n}\end{array}\right]$$
be a row-wise interpretation of the matrix $\M_n$ of type $(\hom(S,X_1), \dots, \hom(S, X_n))$ where $m = \binom{n}{2}$. Each column of the matrix $M'$ above determines a (unique) morphism $S \rightarrow X_1 \times \cdots \times X_n$, and we will write $c'_{i,j}\colon S \rightarrow X_1 \times \cdots \times X_n$ for the morphism corresponding to the left column $c_{i,j}$ of~$\M_n$. We write $y=(y_1,\dots,y_n)\colon S \rightarrow X_1 \times \cdots \times X_n$ for the morphism determined by the right column of~$M'$. We suppose that each $c'_{i,j}$ factorises through $r$ and we must show that $y$ also does. As earlier, we write $\pi_{i,j}\colon X_1 \times \cdots \times X_n \rightarrow X_i \times X_j$ for the two-fold projection determined by $\pi_i$ and~$\pi_j$. It may then be seen that $\pi_{i,j} c'_{i,j} = (y_i, y_j)$, since $M'$ is a row-wise interpretation of~$\M_n$. Therefore $\pi_{i,j}y=\pi_{i,j} c'_{i,j} \c \pi_{i,j}r$ for all $i,j \in \{1,2,\dots,n\}$ with $i<j$ so that $y \c r$ by~$(M_n)$. 

Conversely, suppose that $\C$ has $\M_n$-closed relations and that we are given a monomorphism $r\colon R \rightarrowtail X_1 \times \cdots \times X_n$ and a morphism $y=(y_1,\dots,y_n)\colon S \rightarrow X_1 \times \cdots \times X_n$ such that $\pi_{i,j}y \c \pi_{i,j}r$ for each $i,j\in\{1,2,\dots,n\}$. Thus, there are factorisations $f'_{i,j}\colon S \rightarrow R$ such that $\pi_{i,j}r f'_{i,j} = (y_i, y_j)$. Form the matrix $M'$ whose left columns are determined by the morphisms $c'_{i,j} = rf'_{i,j}$ (again ordering them via the lexicographic order on $\mathbb{N}^2$) and whose right column is determined by the morphism~$y$. This matrix $M'$ is then a row-wise interpretation of $\M_n$ of type $(\hom(S,X_1), \dots, \hom(S,X_n))$. Since $\C$ has $\M_n$-closed relations, we deduce that $y \c r$.
\end{proof}

The matrix $\M_4$ is identical (up to replacement of variables) to the matrix in Remark~2.4 of~\cite{HJJW23}. For this reason, we know we do \emph{not} have $\M_3 \implies_{\lex} \M_4$ although we do have $\M_4 \implies_{\lex} \M_3$. We can actually generalise this.

\begin{proposition}\label{proposition: strict implications lex context}
We have a sequence of strict implications
$$\cdots \implies_{\lex}\M_{n+1} \implies_{\lex} \M_n \implies_{\lex} \cdots \implies_{\lex} \M_5 \implies_{\lex} \M_4 \implies_{\lex} \M_3$$
and all these matrices are non-degenerate.
\end{proposition}

\begin{proof}
According to Theorem~2.5 in~\cite{HJJ22}, a matrix is anti-trivial if and only if its right column appears among its left columns. Clearly, this is not the case for these matrices~$\M_n$. Moreover, we can immediately deduce from Theorem~2.3 in~\cite{HJJ22} that these matrices are not trivial neither and so not degenerate.

Let us now prove that for $n\geqslant 3$, we have $\M_{n+1} \implies_{\lex} \M_n$. To fix notation, let $m=\binom{n}{2}$, $k=\binom{n-1}{2}+1$, $m'=\binom{n+1}{2}$ and $k'=\binom{n}{2}+1$ so that $\M_n\in\Mclex[n,m,k]$ and $\M_{n+1}\in\Mclex[{n+1},m',k']$. According to the algorithm from~\cite{HJJ22}, to prove $\M_{n+1} \implies_{\lex} \M_n$ it is enough to prove that the matrix $M'$ formed by the first $n$ rows of $\M_{n+1}$ admits a row-wise interpretation of type $(\{x_1,\dots,x_k\},\dots,\{x_1,\dots,x_k\})$ whose left columns can be found among the left columns of $\M_n$ and whose right column is the right column of~$\M_n$. This can be easily seen by interpreting in $M'$ each $x_1$ by $x_1$ and each other variable (which appears only once in each row) by the only variable such that the left column $c'_{i,j}$ of $M'$ for $1\leqslant i<j\leqslant n$ is interpreted as the left column $c_{i,j}$ of~$\M_n$, and the left column $c'_{i,n+1}$ of $M'$ for $1\leqslant i\leqslant n$ is interpreted as the left column $c_{i,n}$ if $i<n$ or $c_{1,n}$ if $i=n$ of~$\M_n$.

It remains to prove that we do not have $\M_n \implies_{\lex} \M_{n+1}$. Again according to the algorithm from~\cite{HJJ22}, it is enough to prove that for any matrix $M'$ with $n+1$ rows and whose each row is a row of~$\M_n$, and for any row-wise interpretation $M''$ of $M'$ of type $(\{x_1,\dots,x_{k'}\},\dots,\{x_1,\dots,x_{k'}\})$, if the left columns of $M''$ can be found among the left columns of~$\M_{n+1}$, then so can its right column. By contradiction, suppose $M'$ and $M''$ are such matrices such that the left columns of~$M''$, but not its right column, are among the left columns of~$\M_{n+1}$. In each row of~$\M_n$ (and so of~$M'$), there are exactly $n-1$ many $x_1$'s in its left part (i.e., not counting its rightmost entry). If, for some $j\in\{1,\dots,n+1\}$, these $x_1$'s in the $j^{\textrm{th}}$ row of $M'$ are interpreted in $M''$ as $x_l$ with $1<l\leqslant k'$, then the left columns of $M'$ containing $x_1$ in the $j^{\textrm{th}}$ row are all interpreted the same in $M''$ since $x_l$ appears exactly once in the $j^{\textrm{th}}$ row of $\M_{n+1}$. In that case, since for each $j'\in\{1,\dots,n+1\}$, there is a left column of $M'$ with $x_1$ as $j^{\textrm{th}}$ and $j'^{\textrm{ th}}$ entries, the right column of~$M'$ (constituted only of $x_1$'s) is interpreted in $M''$ in the same way as its left columns containing $x_1$ in the $j^{\textrm{th}}$ row. This would imply that the right column of $M''$ must be found among the left columns of $\M_{n+1}$, which is a contradiction. Therefore, each $x_1$ in the left part of $M'$ is interpreted as $x_1$ in~$M''$. Thus, there are at least $(n+1)(n-1)$ many $x_1$'s in the left part of~$M''$. However, each left column of $M''$ being a left column of $\M_{n+1}$, they contain exactly two $x_1$'s each. There are thus exactly $2m=2\binom{n}{2}=n(n-1)$ many $x_1$'s in the left part of~$M''$, which is a contradiction.
\end{proof}

Let us recall now that from Corollary~2.4 in~\cite{HJJ22} and Corollary~2.4 in~\cite{HJ23} we have the following two propositions.

\begin{proposition}\cite{HJJ22}\label{proposition: non-trivial two row matrix is Mal'tsev}
If $\M\in\matr(2,m,k)$ is a two-row matrix (for integers $m\geqslant 0$ and $k>0$), then $\mclex\{\M\}$ is trivial, anti-trivial or the matrix class of Mal'tsev categories. 
\end{proposition}

\begin{proposition}\cite{HJ23}\label{proposition: M implies Mal}
For any matrix $\M\in\matr$, the implication $\M \implies_{\lex} \Mal$ does \emph{not} hold if and only if every selection of two rows from $\M$ forms an anti-trivial matrix.
\end{proposition}

We are now able to prove the following.

\begin{proposition} \label{proposition: mal-maj lex division}
Given integers $n\geqslant 3$, $m\geqslant 0$ and $k>0$ and a matrix $\M\in\matr(n,m,k)$, we have that either $\M \implies_{\lex} \Mal$ or $\M_n \implies_{\lex} \M$, and these two implications cannot occur simultaneously.
\end{proposition}

\begin{proof}
By construction, every selection of two rows from $\M_n$ forms an anti-trivial matrix and so $\M_n\implies_{\lex}\Mal$ does not hold by Proposition~\ref{proposition: M implies Mal}. This already proves that the two implications of the statement cannot occur simultaneously. Suppose now that $\M \implies_{\lex} \Mal$ does not hold. Then every selection of two rows of $\M$ forms an anti-trivial matrix by Proposition~\ref{proposition: M implies Mal}. Let $\C$ be any finitely complete category in $\mclex\{\M_n\}$ and let us prove it is in $\mclex\{\M\}$. Let $r\colon R \rightarrowtail X_1 \times \cdots \times X_n$ be any relation in $\C$ and let
$$M' = \left[ \begin{array}{ccc|c}
c'_1 & \cdots & c'_m & y
\end{array} \right]$$
be any row-wise interpretation of $\M$ of type $(\hom(S,X_1), \dots, \hom(S,X_n))$ where the left columns $c'_1,\dots,c'_m$ of $M'$ (viewed as morphisms $S \rightarrow X_1 \times \cdots \times X_n$) satisfy $c'_l \c r$ for each $l\in\{1,\dots,m\}$. We must show that the right column $y$ of~$M'$ (also viewed as a morphism $S \rightarrow X_1 \times \cdots \times X_n$) satisfy $y \c r$. Since, by Proposition~\ref{proposition: Mn}, $\C$ satisfies the property~$(M_n)$, we only have to show that $\pi_{i,j} y \c \pi_{i,j} r$ for any $i,j\in\{1,\dots,n\}$. For such $i$ and~$j$, since the matrix obtained by selecting the $i^{\textrm{th}}$ and the $j^{\textrm{th}}$ row of $\M$ is anti-trivial, we know that $\pi_{i,j} y = \pi_{i,j} c'_l$ for some $l\in\{1,\dots,m\}$. Therefore, $\pi_{i,j} y=\pi_{i,j} c'_l \c \pi_{i,j} r$ as desired.
\end{proof}

\section{The regular context}

Recall that a category is \emph{regular}~\cite{BGO71} if it is finitely complete, has coequalisers of kernel pairs and regular epimorphisms are stable under pullbacks. In that case, each morphism factorises as a regular epimorphism followed by a monomorphism. Although, in the finitely complete context, the implications of Proposition~\ref{proposition: strict implications lex context} are strict, this is not the case any more in the regular context as attested by the following theorem.

\begin{theorem}\label{theorem: maj implies Mn}
We have a sequence of equivalences
$$\cdots \equivalent_{\reg}\M_{n+1} \equivalent_{\reg} \M_n \equivalent_{\reg} \cdots \equivalent_{\reg} \M_5 \equivalent_{\reg} \M_4 \equivalent_{\reg} \M_3$$
i.e., for any integer $n\geqslant 3$, in the regular context, the property $(M_n)$ is equivalent to the majority property~$(M_3)$.
\end{theorem}

\begin{proof}
In view of Proposition~\ref{proposition: strict implications lex context} and since $\M_3=\Maj$, it is enough to prove that any regular majority category $\C$ satisfies $(M_n)$ for each $n\geqslant 3$. Given such $\C$ and~$n$, let $r\colon R\rightarrowtail X_1 \times \cdots \times X_n$ be any monomorphism and let $y\colon S\rightarrow X_1 \times \cdots \times X_n$ be any morphism in $\C$ such that, for any $i,j\in\{1,\dots,n\}$, $\pi_{i,j}y \c \pi_{i,j} r$. We must show that $y \c r$. For all $i,j\in\{1,\dots,n\}$, let
$$\xymatrix{R \ar@{->>}[r]^-{e_{i,j}} & R_{i,j} \ar@{ >->}[r]^-{m_{i,j}} & X_i \times X_j}$$ 
be the (regular epimorphism, monomorphism)-factorisation of the composite 
$$\xymatrix{R \ar@{ >->}[r]^-{r} & X_1 \times \cdots \times X_n \ar[r]^-{\pi_{i,j}} & X_i \times X_j.}$$
By the equivalence of (i) and (v) of Theorem~5.1 in~\cite{Hoefnagel20}, and Proposition~4.1 in~\cite{Hoefnagel20}, the diagram
$$\xymatrix@R=3pc@C=7pc{R \ar[r]^-{(e_{i,j})_{(i,j)\in\{1,\dots,n\}^2}} \ar@{ >->}[d]_-{r} & \prod\limits_{i,j = 1}^n R_{i,j} \ar@{ >->}[d]^{\prod\limits_{i,j = 1}^n m_{i,j}} \\
\prod\limits_{l = 1}^n X_l \ar[r]_-{(\pi_{i,j})_{(i,j)\in\{1,\dots,n\}^2}}  & \prod\limits_{i,j = 1}^n X_i \times X_j}$$
is a pullback. Since, for all $i,j\in\{1,\dots,n\}$, we have supposed $\pi_{i,j}y \c \pi_{i,j} r$, we know that $\pi_{i,j}y \c m_{i,j}$; producing a morphism $S\rightarrow \prod_{i,j = 1}^n R_{i,j}$ so that, by the universal property of the above pullback, we get $y \c r$.
\end{proof}

\begin{corollary}\label{corollary: implies Mal or is implied by Maj}
For any matrix $\M$ in $\matr$, we have that either $\M \implies_{\reg} \Mal$ or $\Maj \implies_{\reg} \M$, and these two implications cannot occur simultaneously.
\end{corollary}

\begin{proof}
If these two implications occur simultaneously, we would have $\Maj \implies_{\reg} \Mal$, which, by Theorem~2.3 in~\cite{HJ23}, is equivalent to $\Maj \implies_{\lex} \Mal$. But this last implication does not hold, for instance by Proposition~\ref{proposition: M implies Mal}.

Let us now prove that $\M \implies_{\reg} \Mal$ or $\Maj \implies_{\reg} \M$. If $\M$ is trivial, then $\M \implies_{\reg} \Mal$; and if $\M$ is anti-trivial, we have $\Maj \implies_{\reg} \M$. It remains to treat the case where $\M$ is non-degenerate. Suppose that $\M$ has $n$ rows. If $n\leqslant 2$, by Proposition~\ref{proposition: non-trivial two row matrix is Mal'tsev}, $\M\equivalent_{\lex}\Mal$ and so in particular $\M \implies_{\reg} \Mal$. Let us now suppose that $n\geqslant 3$. Then, we have that either $\M \implies_{\lex} \Mal$ or $\M_n \implies_{\lex} \M$ by  Proposition~\ref{proposition: mal-maj lex division}. The result then follows since $\Maj \equivalent_{\reg} \M_n$ for all $n\geqslant 3$ by Theorem~\ref{theorem: maj implies Mn}.
\end{proof}

\begin{remark}
Let us make clear here that, although Corollary~\ref{corollary: implies Mal or is implied by Maj} takes place in the regular context, it holds for the matrix properties in the sense of~\cite{Janelidze06} but \emph{not} for the matrix properties in the sense of~\cite{Janelidze08}. Indeed, an example of these latter matrix properties is the property of being a Goursat category~\cite{CKP93}. An example of a (regular) Goursat category which is not Mal'tsev is given by the category of implication algebras~\cite{Mitschke71} while an example of a majority category which is not Goursat is given by the category of lattices~\cite{CKP93,Hoefnagel19}.
\end{remark}

The proof of the theorem below makes use of the fact that $\Ari \implies_{\lex} \Mal$ and $\Ari \implies_{\lex} \Maj$. These implications already appear in Figure~\ref{figure 3,7,2}. For a proof of them, we refer the reader to Section~5 of~\cite{HJJ22}. There, it is actually shown that $\mclex\{\Ari\}=\mclex\{\Mal\}\cap\mclex\{\Maj\}$, i.e., a finitely complete category is arithmetical if and only if it is both Mal'tsev and majority.

\begin{theorem}\label{theorem: Ari implies every non trivial matrix}
For any non-trivial matrix $\M$ in $\matr$, we have $\Ari \implies_{\reg} \M$.
\end{theorem}

\begin{proof}
Suppose that $\M$ is any non-trivial matrix with $n>0$ rows. If $n\leqslant 2$, then $\Mal\implies_{\lex}\M$ by Proposition~\ref{proposition: non-trivial two row matrix is Mal'tsev} and therefore $\Ari\implies_{\lex} M$ since $\Ari\implies_{\lex}\Mal$. We can therefore suppose without loss of generality that $n\geqslant 3$. Let $\C$ be a regular arithmetical category, $r\colon R\rightarrowtail X_1 \times \cdots \times X_n$ be any monomorphism in $\C$ and the matrix
$$M'=\left[ \begin{array}{ccc|c}
x_{11} & \cdots & x_{1m} & y_{1} \\
\vdots &  		& \vdots & \vdots\\
x_{n1} & \cdots & x_{nm} & y_{n}
\end{array}\right]$$
be a row-wise interpretation of the matrix $\M$ of type $(\hom(S,X_1), \dots, \hom(S, X_n))$. Viewing each of the left columns $c_1,\dots,c_m$ of $M'$ and its right column $y$ as morphisms $S\rightarrow X_1 \times \cdots \times X_n$, we suppose that $c_l \c r$ for each $l\in\{1,\dots,m\}$ and we must show $y \c r$. For any $i,j\in\{1,\dots,n\}$, consider the (regular epimorphism, monomorphism)-factorisation
$$\xymatrix{R \ar@{->>}[r]^-{e_{i,j}} & R_{i,j} \ar@{ >->}[r]^-{m_{i,j}} & X_i \times X_j}$$
of the composite
$$\xymatrix{R \ar@{ >->}[r]^-{r} & X_1 \times \cdots \times X_n \ar[r]^-{\pi_{i,j}} & X_i \times X_j.}$$
For any such~$i,j$, by Proposition~1.7 in~\cite{Janelidze06b}, since $\M$ is non-trivial, the matrix $\M[i\colon\!\!,j\colon\!\!]$ formed from selecting the $i^{\textrm{th}}$ and the $j^{\textrm{th}}$ row of $\M$ is non-trivial. Hence, $\Mal\implies_{\lex}\M[i\colon\!\!,j\colon\!\!]$ by Proposition~\ref{proposition: non-trivial two row matrix is Mal'tsev} and so $\Ari\implies_{\lex}\M[i\colon\!\!,j\colon\!\!]$. Since $\pi_{i,j}c_l\c m_{i,j}$ for each $l\in\{1,\dots,m\}$, this implies that for any $i,j$ there is a morphism $f_{i,j}\colon S \rightarrow R_{i,j}$ such that $m_{i,j}f_{i,j} = \pi_{i,j} y$. Considering the pullback
$$\xymatrix{Q_{i,j} \ar[r]^-{\beta_{i,j}} \ar@{->>}[d]_-{\alpha_{i,j}} & R \ar@{->>}[d]^-{e_{i,j}} \\ S \ar[r]_-{f_{i,j}} & R_{i,j}}$$
we know that $\pi_{i,j}y\alpha_{i,j}\c \pi_{i,j}r$. Taking the limit of the diagram formed by the $\alpha_{i,j}$ produces a regular epimorphism $\alpha\colon Q \twoheadrightarrow S$ such that, for each $i,j\in\{1,\dots,n\}$, we have $\pi_{i,j}y \alpha \c \pi_{i,j} r$. By Theorem~\ref{theorem: maj implies Mn} we have $\Maj \equivalent_{\reg} \M_n$, so that $\Ari\implies_{\reg}\M_n$. It follows that $y \alpha \c r$, and since $\alpha$ is a regular epimorphism and $r$ a monomorphism, we have that $y \c r$.
\end{proof}

A regular Mal'tsev category $\C$ has been shown to have distributive lattices of equivalence relations if and only if it is a majority category~\cite{Hoefnagel20}. Various alternative characterisations of \emph{equivalence distributive} regular Mal'tsev categories have been given in~\cite{GRT20}, as well as equivalence distributive Goursat categories. By the theorem above, every regular equivalence distributive Mal'tsev category satisfies each non-trivial matrix property from~$\Mclex$. 

By definition, a matrix $\M\in\matr$ is trivial if any finitely complete category with $\M$-closed relations is a preorder. In~\cite{HJ22}, it is proved that we can equivalently consider only finitely complete \emph{pointed} categories to decide whether such a matrix is trivial, i.e., a matrix $\M\in\matr$ is trivial if and only if any finitely complete pointed category with $\M$-closed relations is a preorder. As a corollary of Theorem~\ref{theorem: Ari implies every non trivial matrix}, we prove that we can equivalently only consider varieties of universal algebras.

\begin{corollary}
For a matrix $\M\in\matr$, the following statements are equivalent.
\begin{enumerate}[label=(\roman*)]
\item\label{corollary trivial lex} $\M$ is trivial, i.e., any finitely complete category with $\M$-closed relations is a preorder.
\item\label{corollary trivial pointed lex} Any pointed finitely complete category with $\M$-closed relations is a preorder.
\item\label{corollary trivial reg} Any regular category with $\M$-closed relations is a preorder.
\item\label{corollary trivial varieties} Any variety with $\M$-closed relations is a preorder, i.e., the equation $x=y$ holds in its theory.
\item\label{corollary trivial Set op} The dual of the category of sets $\Set^{\op}$ does not have $\M$-closed relations.
\item\label{corollary trivial pointed Set op} The dual of the category of pointed sets $\Set_{\ast}^{\op}$ does not have $\M$-closed relations.
\item\label{corollary trivial Bool} The category of Boolean algebras $\Bool$ does not have $\M$-closed relations.
\end{enumerate}
\end{corollary}

\begin{proof}
The equivalence \ref{corollary trivial lex}$\Leftrightarrow$\ref{corollary trivial Set op} appears in~\cite{HJJ22} while the equivalences \ref{corollary trivial lex}$\Leftrightarrow$\ref{corollary trivial pointed lex}$\Leftrightarrow$\ref{corollary trivial pointed Set op} appear in~\cite{HJ22}. The implications \ref{corollary trivial lex}$\Rightarrow$\ref{corollary trivial reg}$\Rightarrow$\ref{corollary trivial varieties} are obvious and the implication \ref{corollary trivial varieties}$\Rightarrow$\ref{corollary trivial Bool} follows from the fact that $\Bool$ is a variety of universal algebras which is not a preorder. Finally, the implication \ref{corollary trivial Bool}$\Rightarrow$\ref{corollary trivial lex} follows immediately from Theorem~\ref{theorem: Ari implies every non trivial matrix} since $\Bool$ is a regular arithmetical category.
\end{proof}


\end{document}